\documentclass[a4paper, oneside]{amsart}

\usepackage[
paper=a4paper,
text={132mm,205mm},centering
]{geometry}
\usepackage{scalerel}
\usepackage{amssymb}
\usepackage[arrow]{xy}
\usepackage{pb-diagram,pb-xy}
\usepackage{graphicx,float}
\usepackage{hyperref}
\usepackage{xcolor,color}

\iftrue
\makeatletter
\def\@settitle{%
  \vspace*{-0pt}
  \begin{flushleft}%
    \LARGE\bfseries
    \strut\@title\strut
  \end{flushleft}%
}
\def\@setauthors{%
  \begingroup
  \def\thanks{\protect\thanks@warning}%
  \trivlist
  \raggedright
  \large \@topsep27\p@\relax
  \advance\@topsep by -\baselineskip
  \item\relax
  \author@andify\authors
  \def\\{\protect\linebreak}%
  \authors
  \ifx\@empty\contribs
  \else
    ,\penalty-3 \space \@setcontribs
    \@closetoccontribs
  \fi
  \normalfont
 \@setaddresses
  \endtrivlist
  \endgroup
}
\def\@setaddresses{\par
  \nobreak \begingroup
  \small
  \def\author##1{\nobreak\addvspace\smallskipamount}%
  \def\\{\unskip, \ignorespaces}%
  \interlinepenalty\@M
  \def\address##1##2{\begingroup
    \par\addvspace\bigskipamount\noindent
    \@ifnotempty{##1}{(\ignorespaces##1\unskip) }%
    {\ignorespaces##2}\par\endgroup}%
  \def\curraddr##1##2{\begingroup
    \@ifnotempty{##2}{\nobreak\noindent\curraddrname
      \@ifnotempty{##1}{, \ignorespaces##1\unskip}\/:\space
      ##2\par}\endgroup}%
  \def\email##1##2{\begingroup
    \@ifnotempty{##2}{\nobreak\noindent E-mail address%
      \@ifnotempty{##1}{, \ignorespaces##1\unskip}\/:\space
      \ttfamily##2\par}\endgroup}%
  \def\urladdr##1##2{\begingroup
    \def~{\char`\~}%
    \@ifnotempty{##2}{\nobreak\noindent\urladdrname
      \@ifnotempty{##1}{, \ignorespaces##1\unskip}\/:\space
      \ttfamily##2\par}\endgroup}%
  \addresses
  \endgroup
  \global\let\addresses=\@empty
}
\def\@setabstracta{%
    \ifvoid\abstractbox
  \else
    \skip@17pt \advance\skip@-\lastskip
    \advance\skip@-\baselineskip \vskip\skip@
    \box\abstractbox
    \prevdepth\z@ 
    \vskip-28pt
  \fi
}
\renewenvironment{abstract}{%
  \ifx\maketitle\relax
    \ClassWarning{\@classname}{Abstract should precede
      \protect\maketitle\space in AMS document classes; reported}%
  \fi
  \global\setbox\abstractbox=\vtop \bgroup
    \normalfont\small
    \list{}{\labelwidth\z@
      \leftmargin0pc \rightmargin\leftmargin
      \listparindent\normalparindent \itemindent\z@
      \parsep\z@ \@plus\p@
      
    }%
    \item[\hskip\labelsep\bfseries\abstractname.]%
}{%
  \endlist\egroup
  \ifx\@setabstract\relax \@setabstracta \fi
}

\def\ps@headings{\ps@empty
  \def\@evenhead{%
    \setTrue{runhead}%
    \normalfont\scriptsize
    \rlap{\thepage}\hfill
    \def\thanks{\protect\thanks@warning}%
    \leftmark{}{}}%
  \def\@oddhead{%
    \setTrue{runhead}%
    \normalfont\scriptsize
    \def\thanks{\protect\thanks@warning}%
    \rightmark{}{}\hfill \llap{\thepage}}%
  \let\@mkboth\markboth
}\ps@headings

\def\section{\@startsection{section}{1}%
  \z@{-1.4\linespacing\@plus-.5\linespacing}{.8\linespacing}%
  {\normalfont\bfseries\Large}}
\def\subsection{\@startsection{subsection}{2}%
  \z@{-.8\linespacing\@plus-.3\linespacing}{.5\linespacing\@plus.2\linespacing}%
  {\normalfont\bfseries\large}}
\def\subsubsection{\@startsection{subsubsection}{3}%
  \z@{.7\linespacing\@plus.2\linespacing}{-1.5ex}%
  {\normalfont\bfseries}}
\def\@secnumfont{\bfseries}

\renewcommand\contentsnamefont{\bfseries}
\def\@starttoc#1#2{\begingroup
  \setTrue{#1}%
  \par\removelastskip\vskip\z@skip
  \@startsection{}\@M\z@{\linespacing\@plus\linespacing}%
    {.5\linespacing}{
      \contentsnamefont}{#2}%
  \ifx\contentsname#2%
  \else \addcontentsline{toc}{section}{#2}\fi
  \makeatletter
  \@input{\jobname.#1}%
  \if@filesw
    \@xp\newwrite\csname tf@#1\endcsname
    \immediate\@xp\openout\csname tf@#1\endcsname \jobname.#1\relax
  \fi
  \global\@nobreakfalse \endgroup
  \addvspace{32\p@\@plus14\p@}%
  \let\tableofcontents\relax
}
\def\contentsname{Contents}
\def\l@section{\@tocline{2}{.5ex}{0mm}{5pc}{}}
\def\l@subsection{\@tocline{2}{0pt}{2em}{5pc}{}}
\makeatother
\fi 

\theoremstyle{plain}
\newtheorem{theorem}{Theorem}[section]
\newtheorem{proposition}[theorem]{Proposition}
\newtheorem{corollary}[theorem]{Corollary}

\theoremstyle{definition}

\newtheorem*{problem}{Problem}

\makeatletter
\def\Nopagebreak{\@nobreaktrue\nopagebreak}
\makeatother

\def\Z{\mathbb{Z}}
\def\Q{\mathbb{Q}}

\def\N{\mathbb{N}}

\def\M{\mathcal{M}}

\def\b{\mathrm{b}}

\def\T{\mathcal{T}}

\def\H{\mathcal{H}}

\def\hF{\widehat{F}}

\def\l{\lambda}

\def\exp{\operatorname{exp}}
\def\Aut{\operatorname{Aut}}
\def\Ker{\operatorname{Ker}}

\def\to{\mathchoice{\longrightarrow}{\rightarrow}{\rightarrow}{\rightarrow}}
\makeatletter
\newcommand{\shortxra}[2][]{\ext@arrow 0359\rightarrowfill@{#1}{#2}}
\def\longrightarrowfill@{\arrowfill@\relbar\relbar\longrightarrow}
\newcommand{\longxra}[2][]{\ext@arrow 0359\longrightarrowfill@{#1}{#2}}
\renewcommand{\xrightarrow}[2][]{\mathchoice{\longxra[#1]{#2}}%
  {\shortxra[#1]{#2}}{\shortxra[#1]{#2}}{\shortxra[#1]{#2}}}

\makeatother

\makeatletter
\def\Nopagebreak{\@nobreaktrue\nopagebreak}
\makeatother

\begin{document}

\title
{Lower central subgroups of a free group and its subgroup}
\author{Minkyoung Song}

\begin{abstract}
	For a given free group $F$ of arbitrary rank (possibly infinite), and its subgroup $G$, we address the question whether a lower central subgroup of $G$ can contain a lower central subgroup of $F$. We show that the answer is no if $G$ does not normally generate $F$. The question comes from a study of Hirzebruch-type invariants from iterated $p$-covers for 3-dimensional homology cylinders.
\end{abstract}

\maketitle

\setcounter{tocdepth}{2}

\section{Introduction}
For a group $G$, denote by $G_m$ the $m$th term of the lower central series of $G$, defined inductively by $G_1=G$, $G_{m+1}=[G_m, G]$ for each $m\geq 1$. 

If $F$ is a free group and $G$ its subgroup, then it is obvious that $G_m$ is contained in $F_m$ for every $m \geq 1$. In this paper, we investigate the converse relation: whether some $F_m$ is contained in some $G_k$. Note that $\bigcap_m F_m = 0$.
For $k=1$, if $G$ is a normal subgroup of $F$ with abelian $F/G$, then $G_1$ contains $F_m$ for every $m\geq 2$. We can ask if a subgroup $G$ satisfies $G_2 \supseteq F_m$ for a certain large $m$.
As an answer, we prove the following result:
\begin{theorem}
	\label{theorem}
	Let $F$ be a free group and $G$ a subgroup of $F$ whose normal closure is not~$F$.
	Then $G_2$ never contains $F_m$ for any $m \in \N$.
\end{theorem}

This starts from a study of structures of geometric objects.
Let $\Sigma_{g,n}$ be a compact oriented surface of genus $g$ with $n$ boundary components. 
A \emph{homology cylinder} over $\Sigma_{g,n}$ is defined as a homology cobordism between two copies of~$\Sigma_{g,n}$. The set $\H_{g,n}$ of homology cobordism classes of homology cylinders becomes a group under juxtaposition. The group was introduced as an enlargement of the mapping class group by Garoufalidis and Levine \cite{GL,L01}. It is also a generalization of the concordance group of framed string links. 

In \cite{So}, the author studied the structure of $\H_{g,n}$ by defining extended Milnor invariants and Hirzebruch-type invariants for homology cylinders. Throughout this paper, $p$ denotes a prime number.
Hirzebruch-type intersection form defects associated to $p^r$-fold covers are defined by Cha in \cite{C10} to study homology cobordism of closed 3-manifolds and concordance of links.
Let $d$ be a power of~$p$. 
For a CW-complex $X$, a pair of a cover $\tilde{X}$ obtained by taking $p$-covers repeatedly and a homomorphism $\pi_1(\tilde{X})\to \Z_d$ is called a (\emph{$\Z_d$-valued}) \emph{$p$-structure} for~$X$. Here, a \emph{$p$-cover} means a cover of $p$-power degree. The invariant of a $p$-structure for a 3-manifold is the difference between the Witt classes of the $\Q(\zeta_d)$-valued intersection form and the ordinary intersection form of a 4-manifold bounded by $\tilde{X}$ over $\Z_d$, where $\zeta_d=\exp(2\pi\sqrt{-1}/d)$. This lives in the Witt group $L^0(\Q(\zeta_d))$ of nonsingular hermitian forms over $\Q(\zeta_d)$.

The invariants give rise to invariants of a subgroup of string link concordance group, consisting of $\hF$-string links \cite{C09}. We refer \cite[p.897]{C09} for the definition of $\hF$-string link. Remark that $\hF$-(string) links form the largest known class of (string) links with vanishing Milnor invariants; it is a big open problem in link theory whether all (string) links with vanishing Milnor invariants are $\hF$-(string) links. It turned out that the Hirzebruch-type invariants are homomorphisms on the subgroup of $\hF$-string links.

In \cite{So}, a Hirzebruch-type invariant $\l_\T$ is defined for homology cylinders with a $p$-structure $\T$ for $\Sigma_{g,n}$, or equivalently for $\bigvee^{2g+n-1} S^1$. The $p$-structures are classified in \cite{So}; Let $(\tilde{X},\phi\colon \pi_1(\tilde{X})\to\Z_d)$ be a $p$-structure for $X$. For the cover $\hat X$ induced by $\phi$, if $\pi_1 \hat X \supseteq (\pi_1 X)_m$, the $p$-structure is said to be \emph{of order~$m$}.
Every $p$-structure of a finite CW-complex is of order $m$ for some finite $m$; For the proof, see \cite[Lemma~5.3]{So}.
We revealed that when the invariant is defined;
For a $p$-structure $\T$ of order $m$, the invariant $\l_\T$ is defined for (the homology cobordism class of) a homology cylinder if and only if the homology cylinder has vanishing extended Milnor invariants of length~$m$.
Let $\H_{g,n}(m)$ be the subgroup of $\H_{g,n}$ consisting of homology cylinders with vanishing extended Milnor invariants of length $m$ in $\H_{g,n}$. 
For a $p$-structure $\T$ for $\Sigma_{g,n}$ of order $m$, the Hirzebruch-type invariant
$$\l_\T \colon \H_{g,n}(m) \to L^0(\Q(\zeta_d))$$
is well-defined.
A sufficient condition that $\l_\T$ is additive is given in \cite[Theorem~5.12]{So}.
It follows that $\l_\T$ is a homomorphism on $\bigcap_m \H_{g,n}(m)$ for any $p$-structure~$\T$.
Using homomorphisms $\l_\T$, it turned out that the abelianization of $\bigcap_m \H_{g,n}(m)$ contains a subgroup isomorphic to $\Z^\infty$ if $\b_1(\Sigma_{g,n})=2g+n-1>1$ \cite[Theorem~6.7]{So}. 
If we find $m$ such that the $\l_\T$ are homomorphisms on $\H_{g,n}(m)$, then we will obtain that $H_1(\H_{g,n}(m))$ also contains a subgroup isomorphic to $\Z^\infty$.
To find $\l_\T$ which is a homomorphism on $\H_{g,n}(m)$, 
the author extracted the following from the sufficient condition.
\begin{proposition}\cite[Corollary~5.13]{So}
	\label{prop}
	Let $\Sigma=\Sigma_{g,n}$. Suppose $\T= (\tilde \Sigma, \pi_1 \tilde \Sigma \to \Z_d)$ is a $p$-structure for $\Sigma$ of order $m$.
	If $(\pi_1 \hat \Sigma)_2 \supseteq (\pi_1 \Sigma)_m$ for the $\Z_d$-cover $\hat \Sigma$ of $\tilde \Sigma$ then $\T$ gives a homomorphism 
	 $\l_\T\colon \H_{g,n}(m)\to L^0(\Q(\zeta_d)).$
\end{proposition}

This naturally poses the problem to find a $p$-structure $\T$ for $\Sigma$ satisfying the assumption of the proposition. The problem can be interpreted algebraically as follows:

\begin{problem}
	Suppose $F$ is a finitely generated free group. Find a proper subgroup $G$ of $F$ such that there is an ascending chain $G=F_{(k)}\lhd F_{(k-1)} \cdots \lhd F_{(1)} \lhd F_{(0)}=F$ with each $F_{(i)}/F_{(i+1)}$ an abelian $p$-group and $G_2\supseteq F_m$ for some~$m$.
\end{problem}

We can simplify the problem as follows:
\begin{problem}(simple version)
	Suppose $F$ is a finitely generated free group. Find a proper normal subgroup $G$ such that $F/G$ is abelian and $G_2\supseteq F_m$ for some $m$.
\end{problem}

This is equivalent to the following geometric problem which is the core of the original problem: 
\begin{problem}
	Let $X$ be a CW-complex with $\pi_1 X $ free. Find an abelian cover $\tilde{X}$ of $X$ such that the natural map $\pi_1\tilde{X}/(\pi_1\tilde{X})_m \to H_1(\tilde{X})$ factors through $\pi_1\tilde{X}/(\pi_1 X)_m$ for some $m\geq 2$.
\end{problem}

But, we finally obtain non-existence for the above problems as Theorem~\ref{theorem} shows.
That said, it does not mean that there is no homomorphism $\l_\T$ on $\H_{g,n}(m)$ since Proposition~\ref{prop} follows from only a sufficient condition for $\l_\T$ to be additive in \cite[Theorem~5.12]{So}.

Extending the domain of $\l_\T$ as a homomorphism may help study the mapping class groups of surfaces. The restriction of $\H_{g,n}(m)$ on the mapping class group is the \emph{Johnson filtration} $\M_{g,n}[m]:=\Ker\{\M_{g,n}\to\Aut(F/F_m)\}$. In other words, $\H_{g,n}(m) \cap \M_{g,n} = \M_{g,n}[m]$. The subgroups $\M_{g,n}[2]$ and $\M_{g,n}[3]$ are well known as the \emph{Torelli group} and the \emph{Johnson kernel}, respectively.
In 1938, Dehn proved that $\M_{g,n}$ is finitely generated \cite{D}.
In 1983, Johnson proved that $\M_{g,1}[2]$ and its quotient $\M_{g,0}[2]$ are also finitely generated for $g\geq 3$, but it is discovered that $\M_{2,1}[2]$ and $\M_{2,0}[2]=\M_{2,0}[3]$ are infinitely generated by McCullough and Miller \cite{MM} in 1986. Thereby, the question whether $\M_{g,n}[3]$ is finitely generated for $g \geq 3$ has received a lot of attention since 1990s. Just lately, for $n=0,1$, Ershov and He \cite{EH} showed that $\M_{g,n}[3]$ is finitely generated if $g\geq 12$ and $H_1(\M_{g,n}[m])$ is also finitely generated if $m\geq 3, g\geq 8m-12$. Church, Ershov and Putman proved that also for $n=0,1$, $\M_{g,n}[3]$ is finitely generated if $g\geq4$ and $\M_{g,n}[m]$ is finitely generated if $m\geq 4, g\geq 2m-3$ in \cite{CEP}.
It is still open whether $\M_{g,n}[m]$ is finitely generated for general $g$ and~$n$.
The Hirzebruch-type invariants may be used to prove that the abelianization is infinitely generated if we find a homomorphism $\l_\T$ on the higher order Johnson subgroup.

\subsection*{Acknowledgements}
The author thanks Jae Choon Cha for helpful comments. The work was supported by NRF grant 2011-0030044 (SRC-GAIA).

\section{Non-existence of subgroups}

We denote $[x,y]:=xy\bar x \bar y$ where $\bar x$ means $x^{-1}$.

\begin{theorem}
\label{theorem:core}
	Suppose $F$ is a finitely generated free group. Then there is no normal subgroup of $F$ of prime index whose commutator subgroup contains a term of the lower central series of $F$.
\end{theorem}

\begin{proof}
Suppose there is an index $p$ normal subgroup $G$ of $F$ such that the commutator subgroup $[G, G]$ contains $F_m$ for some $m\in \N$. Then $G$ can be considered as the kernel of a surjective homomorphism $F\twoheadrightarrow \Z_p$.

It is enough to show that if $F=\langle x, y \rangle$ and $G=\Ker\{F\xrightarrow{f}\Z_p\}$ where $f(x)=1, f(y)=0$, then $G_2 \nsupseteq F_m$ for all~$m$.

Let $\omega_n:=[\cdots[[x,y],x],\ldots,x]=[x,y,\underbrace{x,\ldots,x}_{n\textrm{ times}}] \in F_{n+2}$ for $n\geq 0$. We claim that $\omega_n \notin G_2$ for every $n\in \N$. Since $\omega_n$ is an element of $G$, our claim is equivalent that $[\omega_n]\neq 0$ in $G/G_2=H_1(G)$.

The subgroup $G=\langle \langle x^p, y \rangle \rangle_F = \langle x^p, y, xy\bar x, x^2 y \bar x^2, \ldots, x^{p-1} y \bar x^{p-1} \rangle$. Let $a:= x^p$ and $b_k:=x^{k-1} y \bar x^{k-1}$ for $k=1,\ldots,p$, then $G= \langle a, b_1, \ldots, b_p \rangle$. Denote by $S$ the free generating set $\{a, b_1,\ldots, b_p\}$.
 
For $\omega \in G$ and $k=1,\ldots, p$, let $P_k(\omega)$ be the sum of the powers of $b_k$ in $\omega$ as a word expressed in $S$. In other words, $P_k(\omega)$ is the power of $[b_k]$ in~$[\omega]\in H_1(G)$. 
We note that 
\begin{equation}
	xa\bar x=a, xb_1\bar x= b_2, \ldots, x b_{p-1} \bar x = b_p, x b_p \bar x = a b_1 \bar a.
\end{equation}
Thus, conjugating any element of $G$ by $x$ preserves the sum of powers of $a$ in a word in $S$.

Since $a$ does not appear in the reduced word of $\omega_0=xy\bar x \bar y = b_2 \bar b_1$, $[\omega_n]=0 \in H_1(G)$ if and only if $P_k(\omega_n)=0$ for all~$k$.
We observe $P_1(\omega_0)=-1, P_2(\omega_0)=1, P_k(\omega_0)=0$ for $k\geq 3$, and $P_k(\omega_{n+1})=P_k([\omega_n, x])=P_k(\omega_n) + P_k(x\bar \omega_n \bar x) = P_k(\omega_n) - P_k(x \omega_n \bar x)=P_k(\omega_n)-P_{k-1}(\omega_n)$. The last equality comes from $(1)$. Hence we obtain

$$
 \underbrace{  \begin{bmatrix}
	 \vphantom{\ddots} P_1(\omega_n) \\ \vphantom{\ddots} P_2(\omega_n) \\ \vphantom{\ddots}P_3(\omega_n) \\ \vphantom{\ddots}\vdots \\ \vphantom{\ddots}P_p(\omega_n)
    \end{bmatrix}}_{\bold{v}_n}
 =  {\underbrace{
    \begin{bmatrix}
    \vphantom{\ddots} 1 &  & & & -1\\
    \vphantom{\ddots} -1 & 1 &   \\
    \vphantom{\ddots} & -1 & 1 & \\
    \vphantom{\ddots} & & \ddots & \ddots &  \\
    \vphantom{\ddots}  & & & -1 & 1
  \end{bmatrix}}_{A}}^n
 \underbrace{   
  \begin{bmatrix}
	\vphantom{\ddots} -1 \\ \vphantom{\ddots}1 \\ \vphantom{\ddots}0 \\ \vphantom{\ddots}\vdots \\ \vphantom{\ddots}0
    \end{bmatrix}}_{\bold{v}_0}
  $$
Let us calculate the eigenvalues of~$A$. Since $\det(A- \lambda I) = (1- \lambda)^p - 1$, 
the eigenvalues $\lambda_j$ of $A$ are $1-\zeta^j$ where $\zeta$ is the $p$-th root of unity $e^{2\pi i/p}$ and $j=1,\ldots, p$.
The corresponding eigenvector $\bold{x}_j$ to the eigenvalue $\lambda_j$ is 
$
  \begin{bmatrix}
	\vphantom{\ddots} 1 \\ \vphantom{\ddots} \zeta^{(p-1)j} \\ \vphantom{\ddots}\vdots \\ \vphantom{\ddots}\zeta^{2j} \\ \vphantom{\ddots} \zeta^j
    \end{bmatrix} .
$

Since the eigenvalues $\lambda_j$ are all distinct, $\bold{x}_j$ are linearly independent. Thus, $\bold{v}_0$ can be expressed as a linear combination of $\bold{x}_j$. Let $\bold{v}_0 = \sum_{i=1}^{p}  \alpha_j \bold x_j$. Note that $\alpha_j$ is nonzero for some $j\neq p$ since $\bold v_0 \neq \alpha \bold x_p$ for any~$\alpha$. Therefore, $\bold v_n= A^n \bold v_0 = \sum _{i=1}^{p} \alpha_j \lambda_j^n \bold x_j $ is nonzero for any $n\geq 1$. In conclusion, $\omega_n$ is not an element of $G_2$, and it implies that $G_2$ does not contain any~$F_m$.
\end{proof}

Note that prime index does not guarantee normality. For instance, there is a non-normal subgroup $\langle a, b^2, ba^2 b, babab \rangle$ of index 3 in $\Z\ast\Z=\langle a, b \rangle$. 

In fact, the same argument holds not only for $p$ prime, but also when $p$ is replaced by an arbitrary integer $>1$.
Hence the theorem also holds not only for index $p$ normal subgroups but also for normal subgroups with finite cyclic factor groups. Moreover, we can extend Theorem~\ref{theorem:core} as follows:

\begin{corollary}
	\label{corollary}
	Let $F$ be a (possibly infinitely generated) free group. Suppose $G$ is a subgroup of $F$ such that there are $H$ and $K$ with $G\leq K \lhd H \leq F$, a nontrivial abelian factor group $H/K$. Then $G_2$ does not contain $F_m$ for any $m \in \N$.
\end{corollary}

\begin{proof}
	First we generalize Theorem~\ref{theorem:core} to a free group of arbitrary rank. Let $G$ be a normal subgroup of index $p$ where $p$ is a prime. We can assume that $\{ x_i~|~i\in I\}$ is a free generating set of $F$ with an index set $I \ni 1,2$ and $G=\Ker\{f\colon F\twoheadrightarrow \Z_p\}$ with $f(x_1)=1$, $f(x_j)=0$ for $j\neq 1 \in I$. Suppose $G_2 \supseteq F_m$ for some $m$. Let $H=\langle x_1, x_2\rangle$, a subgroup of $F$. Then, $H\cap G=\Ker\{f|_H\colon H\twoheadrightarrow \Z_p\}$ is an index $p$ normal subgroup of $H$. But, $(H\cap G)_2 =  H \cap G_2 \supseteq H \cap F_m \supseteq H_m$. It contradicts Theorem~\ref{theorem:core}.
	 
	Now let us extend $G$ to a subgroup of $F$ with $G\leq K \lhd H \leq F$ and nontrivial abelian $H/K$.
	Suppose $G_2 \supseteq F_m$ for some $m \in \N$. Then, $K_2\supseteq G_2 \supseteq F_m \supseteq H_m$. There is a prime index normal subgroup $K'$ of $H$ which contains $K$ since there is an epimorphism of $H/K$ onto a cyclic group of prime order. We have $(K')_2\supseteq K_2\supseteq H_m$, which is a contradiction.
\end{proof}

For instance, if $F/G$ is the alternating group $A_5$, it has abelian subgroups isomorphic to $\Z_2, \Z_3, \Z_5$, so $G$ satisfies the hypothesis of the above corollary.

Lastly, we give a proof of Theorem~\ref{theorem} stated in the introduction:
\begin{proof}[Proof of Theorem~\ref{theorem}]
	Let $K$ be the normal closure of $G$. Every nontrivial group has a nontrivial abelian subgroup, so there is a nontrivial abelian subgroup $H/K$ of $F/K$. Then $G\leq K \lhd H \leq F$ satisfies the hypothesis of Corollary~\ref{corollary}.
Consequently, the conclusion of the corollary holds for every subgroup whose normal closure is not $F$. Hence we obtain Theorem~\ref{theorem}.
\end{proof}


\begin{thebibliography}{99}
\bibitem[Cha09]{C09} J. C. Cha, \emph{Structure of the string link concordance group and Hirzebruch-type invariants}, Indiana Univ. Math. J. 58 (2009), no. 2, 89--927.
\bibitem[Cha10]{C10} J. C. Cha, \emph{Link concordance, homology cobordism, and Hirzebruch-type defects from iterated $p$-covers}, J. Eur. Math. Soc. (JEMS) 12 (2010), no. 3, 555--610. 
\bibitem[CEP17]{CEP} T. Church, M. Ershov and A. Putman, \emph{On finite generation of the Johnson filtrations}, preprint (2017)  \texttt{arXiv:1711.04779}
\bibitem[Den38]{D} M. Dehn, \emph{Die Gruppe der Abbildungsklassen}, Acta Math. 69 (1938), no. 1, 135--206.
\bibitem[EH17]{EH} M. Ershov and S. He, \emph{On finiteness properties of the Johnson filtrations}, preprint (2017)  \texttt{arXiv:1703.04190}
\bibitem[GL05]{GL} S. Garoufalidis and J. P. Levine, \emph{Tree-level invariants of three-manifolds, Massey products and the Johnson homomorphism}, Graphs and patterns in mathematics and theoretical physics, 173--203, Proc. Sympos. Pure Math., 73, Amer. Math. Soc., Providence, RI, 2005.
\bibitem[Joh83]{J83} D. Johnson, \emph{The Structure of the Torelli Group I: A Finite Set of Generators for J}, Ann. of Math. (2) 118 (1983), no. 3, 423--442.
\bibitem[Lev01]{L01} J. P. Levine, \emph{Homology cylinders: an enlargement of the mapping class group}, Algebr. Geom. Topol. 1 (2001), 243--270 (electronic).
\bibitem[MM86]{MM} D. McCullough and A. Miller, \emph{The genus 2 Torelli group is not finitely generated}, Topology Appl. 22 (1986), 43--49.
\bibitem[So16]{So} M. Song, \emph{Invariants and structures of the homology cobordism group of homology cylinders}, Algebr. Geom. Topol. 16 (2016), 899--943
\end{thebibliography}
\end{document}